\documentclass[a4paper,11pt,reqno]{amsart}
\usepackage{amsthm}
\usepackage{amsmath}
\usepackage{enumerate}
\usepackage{amsfonts}
\usepackage{amssymb}
\usepackage{fullpage}
\usepackage{amsmath,amscd}
\usepackage{stmaryrd}
\usepackage{graphicx}
\usepackage{array}
\usepackage{amsmath,amssymb,amsfonts,dsfont}
\usepackage[utf8]{inputenc} 
\usepackage[english]{babel}
\usepackage[T1]{fontenc} 
\usepackage{graphicx}
\usepackage{float}
\usepackage{pdfpages}
\usepackage[a4paper, margin = 3cm, bottom = 3cm]{geometry}
\usepackage{ifpdf}
\usepackage{marginnote}
\usepackage{hyperref}	
\hypersetup{pdfborder=0 0 0, 
	    colorlinks=true,
	    citecolor=black,
	    linkcolor=blue,
	    urlcolor=red,
	    pdfauthor={Ben-Michael Kohli}
	   }

\newcommand{\e}{\mathrm{e}}

\newtheorem{thm}{Theorem}[section]
\newtheorem{cor}[thm]{Corollary}
\newtheorem{prop}[thm]{Proposition}
\newtheorem{lem}[thm]{Lemma}
\theoremstyle{definition}
\newtheorem{defn}[thm]{Definition}
\theoremstyle{remark}
\newtheorem{rem}[thm]{Remark}
\theoremstyle{definition}

\theoremstyle{definition}

\theoremstyle{definition}

\numberwithin{equation}{section} 
\title{On the Links-Gould invariant and the square of the Alexander polynomial}
\author{Ben-Michael Kohli} 
\address[Ben-Michael Kohli]{IMB UMR5584, CNRS, Universit\'e Bourgogne Franche-Comt\'e, F-21000 Dijon, France.}
\email{Ben-Michael.Kohli@u-bourgogne.fr}
\begin{document}
\begin{abstract}
This paper gives a connection between well chosen reductions of the Links-Gould invariants of oriented links and powers of the Alexander-Conway polynomial. We prove these formulas by showing the representations of the braid groups we derive the specialized Links-Gould polynomials from can be seen as exterior powers of copies of Burau representations.
\end{abstract}
\maketitle
\setcounter{tocdepth}{1}
\tableofcontents
%
%
\section*{Introduction}

The Links-Gould invariants of oriented links $LG^{m,n}(L;t_0,t_1)$ are two variable polynomial quantum invariants. In \cite{DWiIsLi}, David De Wit, Atsushi Ishii and Jon Links proved the following equalities :
$$LG^{1,n}(L;t_0,e^{2i\pi/n} t_0^{-1}) = \Delta_L(t_0^n) $$
where $\Delta_L(t)$ is the Alexander-Conway invariant of $L$. So the Links-Gould invariants contain some topological information. We reinforce that statement by proving the following identity, that the authors we just cited had already conjectured and proved for particular links :
$$ LG^{n,1}(L;t_0,t_0^{-1})= \Delta_L(t_0)^n$$ 
when $n=2,3$. There is no known set of complete skein relations for the square of the Alexander polynomial, so the ideas used in \cite{DWiIsLi} cannot be transposed to our case easily. On the other hand, Ivan Marin and Emmanuel Wagner give a complete set of skein relations for $LG^{2,1}$ in \cite{MarWag}. So evaluating them and testing whether the square of the Alexander polynomial satisfies these evaluated skein relations or not is a possible strategy. However, the cubic skein relation is barely practicable, and such an approach can not be generalized to $n$ greater than $2$.

Our strategy will be to use the robustness of the braid structure to encode links. We express the Alexander-Conway polynomial as a quantum trace as it is done in \cite{Oht}, appendix C. Then we prove the R-matrix representation of braid group $B_n$ used to define reduced Links-Gould invariant $LG^{2,1}$ (resp. $LG^{3,1}$) is isomorphic to the exterior power of a direct sum of Burau representations. That way, the specialized Links-Gould invariants can be written as products of terms, each of which can be identified with the Alexander polynomial of our link seen as a quantum trace.

Our result along with the one we cited at the beginnig of this introduction can also be thought of as a counterpart to the well known result stating that the Jones polynomial and it's square can both be recovered as evaluations of the two variable Kauffman polynomial. See \cite{Lick}, Proposition 16.6, p. 180.

Let us also mention the work of Nathan Geer and Bertrand Patureau-Mirand who extended the Links-Gould invariant to a multivariable link invariant in the same fashion the multivariable Alexander polynomial arises from it's traditional counterpart \cite{GPM}. We suspect that our results remain true in some sense in that multivariable context. 

In section 1, we recall the definition of the Links-Gould invariant of oriented links, and an expression of the Alexander-Conway polynomial in terms of a partial trace. In section 2 we show that the specialized Links-Gould invariant $LG^{2,1}$ can be written as a product by proving two representations of the braid group are isomorphic. We then identify in section 3 each part of the product with the Alexander-Conway invariant. Section 4 is dedicated to extending the proof to the next Links-Gould invariant $LG^{3,1}$.
\section{Definitions and main result}

\subsection{The Alexander-Conway polynomial}

\begin{defn}

(\textit{Reduced and non-reduced Burau representations of a braid}) \\
Set $\mathbb{K} := \mathbb{C}(t^{\pm \frac{1}{2}})$. Let $W_n = <f_1, \ldots, f_n>$ be a $n$-dimensional $\mathbb{K}$-vector space, and $B_n$ be the braid group on $n$ strands. We denote by $\sigma_1, \ldots , \sigma_{n-1}$ the standard Artin generators of the group. The non-reduced Burau representation $\Psi_{W_n} : B_n \longrightarrow GL(W_n)$ is given by :

$$
\Psi_{W_n}(\sigma_i)(f_j) = \left\{
    \begin{array}{lll}
     (1-t) f_i + t^{1/2} f_{i+1} & \mbox{if } j=i ~, \\
        t^{1/2} f_i & \mbox{if }j=i+1 ~, \\
        f_j & \mbox{otherwise}.
    \end{array}
\right. 
$$\newline
Denote by $\delta_n := t^{-(n-1)/2} f_1 + t^{-(n-2)/2} f_2 + \ldots + t^{-1/2} f_{n-1} +f_n$. One can verify that for any $b \in B_n$, $\Psi_{W_n}(b)(\delta_n) = \delta_n$. Hence the reduced Burau representation $\Psi_{\widehat{W_n}} : B_n \longrightarrow GL(\widehat{W_n})$ is given by :
$$
\Psi_{\widehat{W_n}}(b)(\overline{x}) = \overline{\Psi_{W_n}(b)(x)}
$$
where $\widehat{W_n} := W_n /<\delta_n>$.
\end{defn}

Recall that the Alexander theorem states that any link can be obtained as the closure of a given braid. Moreover, the Markov theorem allows us to define link invariants through braids with closure the link. A possible definition of the classical Alexander link invariant uses that procedure.

\begin{defn}

(\textit{Alexander polynomial of a link through the Burau representation}) \\
The Alexander polynomial of an oriented link $L$ is defined as :
$$
\Delta_L(t) \overset{\bullet}{=} \frac{1-t}{1-t^n} ~ det(I-\Psi_{\widehat{W_n}}(b))
$$
where $b$ is any braid in $B_n$ with closure $L$, and the notation $\overset{\bullet}{=}$ means equality up to multiplication by a unit of $\mathbb{C}[t^{\pm 1}]$. 

\end{defn}

However, it is not this definition of the Alexander polynomial that will be useful to us in the following. Next theorem gives another expression, that will be the one we will consider. In particular, this definition removes the ambiguity that relied in the multiplication by a unit.

\begin{defn}
Let $V$ be a $2$-dimensional $\mathbb{K}$-vector space, and $(e_0, e_1)$ be a basis of $V$. We define a representation $\Psi_{V^{\otimes n}} : B_n \longrightarrow GL(V^{\otimes n})$ of $B_n$ :

$$ \Psi_{V^{\otimes{n}}}(\sigma_i) = id_{V}^{\otimes{i-1}} \otimes R_1 \otimes id_{V}^{\otimes{n-i-1}} $$ 
where $$R_1 = \begin{pmatrix}
   1 & 0 & 0 & 0 \\
   0 & 0 & t^{1/2} & 0 \\
   0 & t^{1/2} & 1-t & 0 \\
   0 & 0 & 0 & -t
\end{pmatrix}~ \in End(V \otimes V)  $$ 
\newline
is an R-matrix, that is a solution of the Yang-Baxter equation.
\end{defn}

\begin{thm} \label{defalexander}
Let $L$ be an oriented link and $b \in B_n$ be any braid with closure $L$. We define
$$
h = \begin{pmatrix}
   t^{1/2} & 0 \\
   0 & -t^{1/2}
\end{pmatrix} \in {End}(V).
$$
Then : \newline
1) $\exists$ $c \in \mathbb{K}$ such that $trace_{2,3,\ldots,n}((id_{V} \otimes h^{\otimes n-1}) \circ \Psi_{V^{\otimes{n}}}(b))= c.id_V$,\newline
2) $c$ is a link invariant and is equal to the Alexander polynomial of $L$, $\Delta_L(t)$.

\end{thm} 

For a detailed proof, see \cite{Oht}, appendix C.

\begin{rem}
R-matrix $R_1$ can be recovered from the universal R-matrix of ribbon Hopf algebra $U_{\zeta}(sl_2)$ at root of unity $\zeta = -1$ thanks to a one parameter family of irreducible representations of $U_{\zeta}(sl_2)$ on $V$. For precise explanations, see \cite{Oht}, p.95-97, or \cite{Viro}. It may also be derived from the quantized universal enveloping algebra of $gl(1|1)$, that is $U_q(gl(1|1))$. See \cite{Viro} or \cite{Resh} for details.

\end{rem}

\begin{rem}
Identifying algebras $End(V^{\otimes n})$ and $End(V)^{\otimes n}$, the partial trace operator verifies $trace_{2,3,\ldots,n}(f_1 \otimes \ldots \otimes f_n) := trace(f_2) trace(f_3) \ldots trace(f_n) f_1 \in End(V)$ for any $f_1, \ldots, f_n \in End(V)$.

\end{rem}

\begin{cor} \label{formuletrace}
With the same notations, this formula follows from theorem \ref{defalexander} :
$$
\Delta_L(t) = \frac{1}{2}~trace((id_{V} \otimes h^{\otimes n-1}) \circ \Psi_{V^{\otimes{n}}}(b)).
$$ 

\end{cor}

\begin{proof}
Applying the trace operator on each side of the formula that defines constant $c$, we obtain :
$$
2c = trace \large( trace_{2,3,\ldots,n}((id_{V} \otimes h^{\otimes n-1}) \circ \Psi_{V^{\otimes{n}}}(b)) \large).
$$
But : $$trace \large( trace_{2,3,\ldots,n}(f_1 \otimes \ldots \otimes f_n) \large) = trace \large( trace(f_2) trace(f_3) \ldots trace(f_n) f_1 \large)$$ 
$$= trace(f_2) trace(f_3) \ldots trace(f_n) trace(f_1) = trace \large( f_1 \otimes \ldots \otimes f_n \large)$$ 
Since the trace and the partial trace are linear maps, we can extend the equality to any $f \in End(V^{\otimes n})\simeq End(V)^{\otimes n}$, which provides the result.
\end{proof}

\subsection{The Links-Gould invariant $LG^{2,1}$ of links} \label{defLG}

\begin{defn}
Set $\mathbb{L} := \mathbb{C}(t_0^{\pm \frac{1}{2}},t_1^{\pm \frac{1}{2}})$. Let $W = <e_1,\ldots, e_4>$ be a four-dimensional $\mathbb{L}$-vector space. The following linear map $R$, expressed in basis $(e_1 \otimes e_1 , e_1 \otimes e_2, e_1 \otimes e_3, e_1 \otimes e_4, e_2 \otimes e_1 , e_2 \otimes e_2, e_2 \otimes e_3, \ldots)$, is an automorphism of $W \otimes W$ and an R-matrix \cite{DeWKauLin}, p.186 :
\setcounter{MaxMatrixCols}{18}

\begin{center}

\scalebox{0.75}{
$
\begin{pmatrix}
   {-t_0} & . & .  &  . & .  & .  & .  & .  & .  & .  & .  & .  & .  & .  & .  & .  \\
   . & . & .  & .  & -t_0^{1/2}  & .  & .  & .  & .  & .  & .  & .  & .  & .  & .  & .  \\
   . & . & .  & .  & .  & .  & .  & .  & -t_0^{1/2}  & .  & .  & .  & .  & .  & .  & .  \\
   . & . & .  & .  & .  & .  & .  & .  & .  & .  & .  & .  & -1  & .  & .  & .  \\
   . & -t_0^{1/2} & .  & .  & 1-t_0  & .  & .  & .  & .  & .  & .  & .  & .  & .  & .  & .  \\
   . & . & .  & .  & .  & 1  & .  & .  & .  & .  & .  & .  & .  & .  & .  & .  \\
   . & . & .  & .  & .  & .  & 1-t_0 t_1  & .  & .  & t_0^{1/2} t_1^{1/2}  & .  & .  & t_0^{1/2} t_1^{1/2} Y  & .  & .  & .  \\
   . & . & .  & .  & .  & .  & .  & .  & .  & .  & .  & .  & .  &  -t_1^{1/2} & .  & .  \\
   . & . & -t_0^{1/2}  & .  & .  & .  & .  & .  & 1-t_0  & .  & .  & .  & .  & .  & .  & .  \\
   . & .  & .  & .  & .  & .  & t_0^{1/2} t_1^{1/2}  & .  & .  & .  & .  & .  & -Y  & .  & .  & .  \\
   . & . & .  & .  & .  & .  & .  & .  & .  & .  & 1  & .  & .  & .  & .  & .  \\
   . & . & .  & .  & .  & .  & .  & .  & .  & .  & .  & .  & .  & .  &  -t_1^{1/2} & .  \\
   . & . & .  & -1  & .  & .  &  t_0^{1/2} t_1^{1/2} Y & .  & .  & -Y  & .  & .  & -Y^{2}  & .  & .  & .  \\
   . & . &  . & .  & .  & .  & .  & -t_1^{1/2}  & .  & .  & .  & .  & .  & 1-t_1  & .  & .  \\
   . & . & .  & .  & .  & .  & .  & .  & .  & .  & .  & -t_1^{1/2}  & .  & .  & 1-t_1  & .   \\
   . & . & .  & .  & .  & .  & .  & .  & .  & .  & .  & .  & .  & .  & .  &  -t_1
\end{pmatrix}
$
}
\end{center}

\noindent where $Y = ((t_0 - 1)(1- t_1))^{1/2}$. \\
We denote by $b_R^n$ the representation of braid group $B_n$ derived from this R-matrix. It is given by the standard formula :
$$
b_R^n(\sigma_i) = id_{W}^{\otimes{i-1}} \otimes R \otimes id_{W}^{\otimes{n-i-1}} \text{ , } i = 1, \ldots, n-1 \text{.}
$$

\end{defn}

\begin{thm}
Let $L$ be an oriented link, and $b \in B_n$ a braid with closure $L$. Define $\mu$ the following linear map :
$$
\mu = \begin{pmatrix}
   t_0^{-1} & . & . & . \\
   . & -t_1 & . & . \\
   . & . & -t_0^{-1} & . \\
   . & . & . & t_1
\end{pmatrix}
\in End(W).
$$
Then : \newline
1) $\exists$ $c \in \mathbb{L}$ such that $trace_{2,3,\ldots,n}((id_{W} \otimes \mu^{\otimes n-1}) \circ b_R^n(b))= c.id_W$,\newline
2) $c$ is an oriented link invariant called Links-Gould invariant of $L$. We will denote it by $LG^{2,1}(L ; t_0 , t_1)$, or simply $LG(L ; t_0 , t_1)$ when it is not ambiguous to do so.
\end{thm}

\begin{rem}
With the notations used in \cite{DeWKauLin}, $LG(L ; q^{-2 \alpha} , q^{2 \alpha +2})$ is the Links-Gould invariant introduced in that paper, using a one parameter family of representations of quantum superalgebra $U_q(gl(2|1))$.
\end{rem}

\begin{rem}
As in corollary \ref{formuletrace}, we explicit a formula for $LG$, that will be useful to us : $$LG(L ; t_0 , t_1)=\frac{1}{4}~trace((id_{W} \otimes \mu^{\otimes n-1}) \circ b_R^n(b)).$$
\end{rem}

\subsection{The conjecture}
The Links-Gould polynomial we just defined is a particular case of a larger family of Links-Gould invariants, introduced by David De Wit in \cite{DeWit}. We will write $LG^{m,n}$, where $m$, $n$ are positive integers. Each invariant is associated with a highest weight $U_q(gl(m|n))$ representation. The invariant we explicited corresponds to case $(2,1)$. In \cite{DWiIsLi}, D. De Wit, A. Ishii and J. Links conjectured that, in their set of variables, well chosen reductions  of $LG^{m,n}$ recover powers of the Alexander-Conway polynomial :
$$
LG^{m,n}( L ; \tau , e^{i \pi / n}) = \Delta_L(\tau^{2 n})^{m}.
$$
In the same paper, they prove the conjecture in cases $(1,n)$, as well as in case $(2,1)$ for a certain class of braids, using representation theory of $U_q(gl(n|1))$. Using a new strategy, we prove the conjecture completely in cases $(2,1)$ and $(3,1)$. We believe that the method can be generalized to cases $(n,1)$ after extended and extensive computation. \\
\\
In the next two sections, we prove case $(2,1)$. We express the conjecture in the set of variables we used to introduce $LG^{2,1}$. We want to prove :
$$
LG^{2,1}(L;\tau,-1) = \Delta_{L}(\tau^2)^2.
$$
Variables $(\tau, q)$ and $(t_0,t_1)$ are related by : $t_1^{1/2} = \tau^{-1}q$, $ t_0^{1/2}=\tau$ (and therefore $t_0^{1/2}t_1^{1/2} = q$). Since in our case $q=-1$, we obtain :
$$
t_0^{1/2}t_1^{1/2} = -1 \text{, } \tau^2 = t_0 = t_1^{-1}.
$$
Thus, once it is formulated in a convenient way, our main result states :
\begin{thm} \label{main}
For any oriented link L, $LG^{2,1}(L; t_0, t_0^{-1}) \overset{\bullet}{=} \Delta_L(t_0)^2$.
\end{thm}

\begin{rem}
Since $LG^{2,1}$ is symmetric in $t_0$ and $t_1$ \cite{DeWKauLin, DeWit}, $LG^{2,1}(t_0,t_0^{-1})$ is symmetric in $t_0$ and $t_0^{-1}$. So if we chose $\Delta_L$ to be the Conway symmetric version of the Alexander polynomial, we are sure the equality is only up to a sign $\pm 1$. In particular, when $L$ is a knot, Ishii shows in \cite{Ish} that $LG^{2,1}(t,1) = LG^{2,1}(1,t) = 1$. So $LG^{2,1}(1,1) = 1 \geqslant 0$. Since $\Delta_L(1)^2 \geqslant 0$, we see that in this case the equality holds. 
\end{rem}

%
%
\section{The reduced Links-Gould invariant expressed as a product}

We derive a representation of the braid group $B_n$ from the Burau representation. We identify it with a specialization of the R-matrix representation given in subsection \ref{defLG}. Then we use this identification to express the specialized Links-Gould invariant as a product.

\subsection{A representation of $B_n$ isomorphic to $b_R^n(t_0,t_0^{-1})$}

Denote by $F$ the following Burau representation of $B_n$ on vector space $W_n = <f_1, \ldots, f_n>$ where we replace $t_0$ by $t_0^{-1}$ :
$$
F(\sigma_i)(f_j) = \left\{
    \begin{array}{lll}
     (1-t_0^{-1}) f_i + t_0^{-1/2} f_{i+1} & \mbox{if } j=i ~, \\
        t_0^{-1/2} f_i & \mbox{if }j=i+1 ~, \\
        f_j & \mbox{otherwise}.
    \end{array}
\right.
$$
In a similar way, let $G$ be the representation of $B_n$ on $n$-dimensional vector space $W_n = <g_1, \ldots, g_n>$ given by :
$$
G(\sigma_i)(g_j) = \left\{
    \begin{array}{lll}
     -t_0^{1/2} g_{i+1} & \mbox{if } j=i ~, \\
        -t_0^{1/2} g_i + (1-t_0) g_{i+1} & \mbox{if }j=i+1 ~, \\
        g_j & \mbox{otherwise}.
    \end{array}
\right.
$$

\begin{prop}
Representation $G$ is isomorphic to the Burau representation of $B_n$.
\end{prop}

\begin{proof}
One can verify that for $i = 1, 2, \ldots, n-1$ : $J_n \circ \Psi_{W_n}(\sigma_i) = G(\sigma_i) \circ J_n$ where $J_n$ can be defined inductively : 
$
J_2 = \begin{pmatrix}
   0 & 1 \\
   -1 & 0
\end{pmatrix}
\text{ and }$ \\
 $$J_n =
\left (\begin{array}{c|c}
  J_{n-1} & {\begin{matrix}
   (-1)^{2}t^{(n-2)/2} \\
   \vdots \\
   (-1)^{n-1}t^{1/2} \\
   (-1)^{n}t^{0/2}
\end{matrix}} \\
\hline
(-1)^{n+1}t^{-(n-2)/2} ~\dots~ (-1)^{n+1}t^{-1/2} \text{  }~ (-1)^{n+1}t^{-0/2}& 0
\end{array}\right).
$$
\\
Moreover, evaluating the determinant of $J_n$, we deduce that $J_n$ is an automorphism. Indeed, $det J_{n+1} = (-1)^{n+1} (t^{1/2} + t^{-1/2}) det J_n + det J_{n-1}$. So $det J_n \in \mathbb{Z}[t^{\pm 1/2}]$ is invertible in $\mathbb{Q}(t^{\pm 1/2})$ since it has degree $n-2$ in both variables $t^{1/2}$ and $t^{-1/2}$.
\end{proof}

\begin{defn}
If $F \oplus G$ is the representation of the $n$-strand braid group on $W_n \oplus W_n$ built from $F$ and $G$, we consider the exterior representation $\Psi_n := \bigwedge (F \oplus G)$ on exterior algebra $\bigwedge(W_n \oplus W_n)$. 
\end{defn}

\begin{rem} Note that $b_R^n(t_0,t_0^{-1})$ and $\Psi_n$ both are $4^n = 2^{2n}-$dimensional representations. 
\end{rem}

We are going to show these two representations are isomorphic. For that we study first the case where $n=2$. Since $t_1 = t_0^{-1}$, we have a simpler R-matrix $R$ :

\begin{center}
\scalebox{0.8}{$
R = \begin{pmatrix}
   {-t_0} & . & .  &  . & .  & .  & .  & .  & .  & .  & .  & .  & .  & .  & .  & .  \\
   . & . & .  & .  & -t_0^{1/2}  & .  & .  & .  & .  & .  & .  & .  & .  & .  & .  & .  \\
   . & . & .  & .  & .  & .  & .  & .  & -t_0^{1/2}  & .  & .  & .  & .  & .  & .  & .  \\
   . & . & .  & .  & .  & .  & .  & .  & .  & .  & .  & .  & -1  & .  & .  & .  \\
   . & -t_0^{1/2} & .  & .  & 1-t_0  & .  & .  & .  & .  & .  & .  & .  & .  & .  & .  & .  \\
   . & . & .  & .  & .  & 1  & .  & .  & .  & .  & .  & .  & .  & .  & .  & .  \\
   . & . & .  & .  & .  & .  & .  & .  & .  & -1  & .  & .  &  -Y  & .  & .  & .  \\
   . & . & .  & .  & .  & .  & .  & .  & .  & .  & .  & .  & .  &  t_0^{-1/2} & .  & .  \\
   . & . & -t_0^{1/2}  & .  & .  & .  & .  & .  & 1-t_0  & .  & .  & .  & .  & .  & .  & .  \\
   . & .  & .  & .  & .  & .  & -1  & .  & .  & .  & .  & .  & -Y  & .  & .  & .  \\
   . & . & .  & .  & .  & .  & .  & .  & .  & .  & 1  & .  & .  & .  & .  & .  \\
   . & . & .  & .  & .  & .  & .  & .  & .  & .  & .  & .  & .  & .  &  t_0^{-1/2} & .  \\
   . & . & .  & -1  & .  & .  &  -Y & .  & .  & -Y  & .  & .  & -Y^{2}  & .  & .  & .  \\
   . & . &  . & .  & .  & .  & .  & t_0^{-1/2}  & .  & .  & .  & .  & .  & 1-t_0^{-1}  & .  & .  \\
   . & . & .  & .  & .  & .  & .  & .  & .  & .  & .  & t_0^{-1/2}  & .  & .  & 1-t_0^{-1}  & .   \\
   . & . & .  & .  & .  & .  & .  & .  & .  & .  & .  & .  & .  & .  & .  &  -t_0^{-1}
\end{pmatrix}
$}
\end{center}
where $Y = t_0^{1/2} - t_0^{-1/2}$. \\
\\
R-matrix $R$ can be rewritten in basis ${\mathcal{B}} = (\vert e_1 \otimes e_1 \vert , \vert e_4 \otimes e_4 \vert , \vert e_2 \otimes e_2 \vert , \vert e_3 \otimes e_3 \vert , \vert e_1 \otimes e_2 , e_2 \otimes e_1 \vert , \vert e_1 \otimes e_3 , e_3 \otimes e_1 \vert , \vert e_3 \otimes e_4 , e_4 \otimes e_3 \vert , \vert e_2 \otimes e_4 , e_4 \otimes e_2 \vert , \vert e_1 \otimes e_4 , e_2 \otimes e_3 , e_3 \otimes e_2 , e_4 \otimes e_1 \vert )$ as follows :

\begin{center}
\scalebox{0.8}{
$
\begin{pmatrix}
   -t_0 & . & .  &  . & .  & .  & .  & .  & .  & .  & .  & .  & .  & .  & .  & .  \\
   . & -t_0^{-1} & .  & .  & .  & .  & .  & .  & .  & .  & .  & .  & .  & .  & .  & .  \\
   . & . & 1  & .  & .  & .  & .  & .  & .  & .  & .  & .  & .  & .  & .  & .  \\
   . & . & .  & 1  & .  & .  & .  & .  & .  & .  & .  & .  & .  & .  & .  & .  \\
   . & . & .  & .  & 0  & -t_0^{1/2}  & .  & .  & .  & .  & .  & .  & .  & .  & .  & .  \\
   . & . & .  & .  & -t_0^{1/2}  & 1-t_0  & .  & .  & .  & .  & .  & .  & .  & .  & .  & .  \\
   . & . & .  & .  & .  & .  & 0  & -t_0^{1/2}  & .  & .  & .  & .  & .  & .  & .  & .  \\
   . & . & .  & .  & .  & .  & -t_0^{1/2}  & 1-t_0  & .  & .  & .  & .  & .  & .  & .  & .  \\
   . & . & .  & .  & .  & .  & .  & .  & 0  & t_0^{-1/2}  & .  & .  & .  & .  & .  & .  \\
   . & .  & . & .  & .  & .  & .  & .  & t_0^{-1/2}  & 1-t_0^{-1}  & .  & .  & .  & .  & .  & .  \\
   . & . & .  & .  & .  & .  & .  & .  & . & . & 0  & t_0^{-1/2}    & .  & .  & .  & .  \\
   . & . & .  & .  & .  & .  & .  & .  & .  & . & t_0^{-1/2}  & 1-t_0^{-1}   & .  & .  &  . & .  \\
   . & . & .  & .  & .  & .  &  . & .  & .  & .  & .  & .  & 0  & 0  & 0  & -1  \\
   . & . &  . & .  & .  & .  & .  & .  & .  & .  & .  & .  & 0  & 0  & -1  & -Y  \\
   . & . & .  & .  & .  & .  & .  & .  & .  & .  & .  & .  & 0  & -1  & 0  & -Y   \\
   . & . & .  & .  & .  & .  & .  & .  & .  & .  & .  & .  & -1  & -Y  & -Y  &  -Y^2
\end{pmatrix}
$}
\end{center}
Family $(f_1,f_2,g_1,g_2)$ is a basis for $W_2 \oplus W_2$. Since $B_2 = <\sigma_1>$, we are looking for a linear automorphism $I$ : $W^{\otimes 2}$ $\longrightarrow$ $\bigwedge(W_2 \oplus W_2)$ such that $\Psi_2(\sigma_1) \circ I = I \circ \underbrace{b_R^2(\sigma_1)}_{R}$. 

\noindent In basis $(f_1 , f_2 , g_1 , g_2)$, $$(F \oplus G) (\sigma_1) = \begin{pmatrix}
1-t_0^{-1} & t_0^{-1/2} &. &. \\
t_0^{-1/2} & 0 & . &. \\
. & . & 0 & -t_0^{1/2} \\
. & . & -t_0^{1/2} & 1-t_0
\end{pmatrix}
$$ 

Therefore, computation of $\Psi_2(\sigma_1)$ shows that in basis ${\mathcal{C}} = (\vert g_1 \wedge g_2 \vert , \vert f_1 \wedge f_2 \vert , \vert 1 \vert , \vert f_1 \wedge f_2 \wedge g_1 \wedge g_2 \vert , \vert g_1 , g_2 \vert , \vert f_2 \wedge g_1 \wedge g_2 , f_1 \wedge g_1 \wedge g_2 \vert , \vert f_1 \wedge f_2 \wedge g_1 , f_1 \wedge f_2 \wedge g_2 \vert , \vert f_2 , f_1 \vert , \vert f_2 \wedge g_1 , f_2 \wedge g_2 , f_1 \wedge g_1 , f_1 \wedge g_2 \vert)$ we obtain the same matrix :
\begin{center}
\scalebox{0.8}{
$
\begin{pmatrix}
   -t_0 & . & .  &  . & .  & .  & .  & .  & .  & .  & .  & .  & .  & .  & .  & .  \\
   . & -t_0^{-1} & .  & .  & .  & .  & .  & .  & .  & .  & .  & .  & .  & .  & .  & .  \\
   . & . & 1  & .  & .  & .  & .  & .  & .  & .  & .  & .  & .  & .  & .  & .  \\
   . & . & .  & 1  & .  & .  & .  & .  & .  & .  & .  & .  & .  & .  & .  & .  \\
   . & . & .  & .  & 0  & -t_0^{1/2}  & .  & .  & .  & .  & .  & .  & .  & .  & .  & .  \\
   . & . & .  & .  & -t_0^{1/2}  & 1-t_0  & .  & .  & .  & .  & .  & .  & .  & .  & .  & .  \\
   . & . & .  & .  & .  & .  & 0  & -t_0^{1/2}  & .  & .  & .  & .  & .  & .  & .  & .  \\
   . & . & .  & .  & .  & .  & -t_0^{1/2}  & 1-t_0  & .  & .  & .  & .  & .  & .  & .  & .  \\
   . & . & .  & .  & .  & .  & .  & .  & 0  & t_0^{-1/2}  & .  & .  & .  & .  & .  & .  \\
   . & .  & . & .  & .  & .  & .  & .  & t_0^{-1/2}  & 1-t_0^{-1}  & .  & .  & .  & .  & .  & .  \\
   . & . & .  & .  & .  & .  & .  & .  & . & . & 0  & t_0^{-1/2}    & .  & .  & .  & .  \\
   . & . & .  & .  & .  & .  & .  & .  & .  & . & t_0^{-1/2}  & 1-t_0^{-1}   & .  & .  &  . & .  \\
   . & . & .  & .  & .  & .  &  . & .  & .  & .  & .  & .  & 0  & 0  & 0  & -1  \\
   . & . &  . & .  & .  & .  & .  & .  & .  & .  & .  & .  & 0  & 0  & -1  & -Y  \\
   . & . & .  & .  & .  & .  & .  & .  & .  & .  & .  & .  & 0  & -1  & 0  & -Y   \\
   . & . & .  & .  & .  & .  & .  & .  & .  & .  & .  & .  & -1  & -Y  & -Y  &  -Y^2
\end{pmatrix}
$}
\end{center}
Setting $I$ : $W^{\otimes 2}$ $\longrightarrow$ $\bigwedge(W_2 \oplus W_2)$ the linear map that transforms ${\mathcal{B}}$ into ${\mathcal{C}}$, we obtain an automorphism that preserves the $\mathbb{C}[B_2]$-module structure :
$$
\Psi_2(\sigma_1) \circ I = I \circ R.
$$

The idea is to generalize that construction for $n$ larger than $2$. We choose the following reference basis for $\bigwedge(W_n \oplus W_n)$ :
$$
(f_{i_1} \wedge \ldots \wedge f_{i_p} \wedge g_{j_1} \wedge \ldots \wedge g_{j_m})_{
1 \leq i_1 < \ldots < i_p \leq n ~,~ 1 \leq j_1 < \ldots < j_m \leq n }.
$$
When we refer to $Reord(u_{i_1} \wedge \ldots \wedge u_{i_r})$, where the $u_{i_k}$ are distinct elements of $\lbrace f_1, \ldots , f_n , g_1 , \ldots , g_n \rbrace $, we mean that we rewrite the element so that it becomes part of the reference basis we just mentioned. \\
\\
We set $
I_1 = \left\{
    \begin{array}{lllll}
     W & \longrightarrow & \bigwedge(W_1 \oplus W_1) \\
        e_1 & \longmapsto & g_1 \\
        e_2 & \longmapsto & 1 \\
        e_3 & \longmapsto & f_1 \wedge g_1\\
        e_4 & \longmapsto & f_1\\
    \end{array}
\right. 
$  and  $
I_2 = \left\{
    \begin{array}{lllll}
     W^{\otimes 2} & \longrightarrow & \bigwedge(W_2 \oplus W_2) \\
        e_i \otimes e_1 & \longmapsto & I_1(e_i) \wedge g_2 \\
        e_i \otimes e_2 & \longmapsto & I_1(e_i) \\
        e_i \otimes e_3 & \longmapsto & Reord(I_1(e_i) \wedge f_2 \wedge g_2)\\
        e_i \otimes e_4 & \longmapsto & Reord(I_1(e_i) \wedge f_2)\\
    \end{array}
\right. 
$. \\
\\
An elementary calculation shows that $I_2 = I$. We can extend these maps by induction setting :
$$
I_n = \left\{
    \begin{array}{cllll}
     W^{\otimes n} & \longrightarrow & \bigwedge(W_n \oplus W_n) \\
        e_{i_1} \otimes \ldots \otimes e_{i_{n-1}} \otimes e_1 & \longmapsto & I_{n-1}(e_{i_1} \otimes \ldots \otimes e_{i_{n-1}}) \wedge g_n \\
        e_{i_1} \otimes \ldots \otimes e_{i_{n-1}} \otimes e_2 & \longmapsto & I_{n-1}(e_{i_1} \otimes \ldots \otimes e_{i_{n-1}}) \\
        e_{i_1} \otimes \ldots \otimes e_{i_{n-1}} \otimes e_3 & \longmapsto & Reord(I_{n-1}(e_{i_1} \otimes \ldots \otimes e_{i_{n-1}}) \wedge f_n \wedge g_n)\\
        e_{i_1} \otimes \ldots \otimes e_{i_{n-1}} \otimes e_4 & \longmapsto & Reord(I_{n-1}(e_{i_1} \otimes \ldots \otimes e_{i_{n-1}}) \wedge f_n)\\
    \end{array}
\right. .
$$
It is easy to see that map $I_n$ sends the natural basis of $W^{\otimes n}$ derived from $(e_1,e_2,e_3,e_4)$ on our reference basis of $\bigwedge(W_n \oplus W_n)$. In particular, $I_n$ is a linear automorphism. Note that map $I_n$ can also be written directly :
$$
I_n ( e_{i_1} \otimes \ldots \otimes e_{i_{n}}) = \big( \underset{k \text{ : } i_k = 3,4}{\bigwedge} f_k \big) \wedge \big( \underset{k \text{ : } i_k = 1,3}{\bigwedge} g_k \big)
$$

\begin{prop} \label{commutation}
Map $I_n$ is a $\mathbb{C}[B_n]$-module automorphism. That is, for any $b\in B_n$ :
$$
\Psi_n(b) \circ I_n = I_n \circ b_R^n(b).
$$
\end{prop}
\begin{proof}
We prove the commutation by induction on $n$. For details, see section \ref{appendix} where we do the necessary computations.
\end{proof}

\subsection{A convenient expression for $LG^{2,1}$}
Now we have built an exterior representation that is isomorphic to $b_R^n(t_0, t_0^{-1})$, we use it to write the reduction of the Links-Gould polynomial as the product of two quantities we will then identify. \\
\\
Using proposition \ref{commutation}, we can write :
$$
LG(L ; t_0 , t_0^{-1}) = \frac{1}{4}~trace(\underbrace{I_n \circ (id_{W} \otimes \mu^{\otimes n-1}) \circ I_n^{-1}}_{\tilde{\mu}} \circ \Psi_n(b)).
$$
We wish to explicit $\tilde{\mu}$.
\begin{lem} \label{mutilde}
Map $\tilde{\mu}$ can be expressed on the reference basis of $\bigwedge(W_n \oplus W_n)$ : 
\begin{align*}
\tilde{\mu}(f_{i_1} \wedge \ldots \wedge f_{i_p} \wedge g_{j_1} \wedge \ldots \wedge g_{j_m}) &= t_0^{-(n-1)} (-1)^{n-1} (-1)^{\# \{ k \in \{ 2, \ldots , n \} | f_k \text{ appears} \}} \\
& (-1)^{\# \{ k \in \{ 2, \ldots , n \} | g_k \text{ appears} \}} f_{i_1} \wedge \ldots \wedge f_{i_p} \wedge g_{j_1} \wedge \ldots \wedge g_{j_m} .
\end{align*}
\end{lem}

\begin{proof}
If $f_{i_1} \wedge \ldots \wedge f_{i_p} \wedge g_{j_1} \wedge \ldots \wedge g_{j_m}$ is an element of the basis of $\bigwedge(W_n \oplus W_n)$, we denote by $e_{l_1} \otimes \ldots \otimes e_{l_n}$ it's image under $I_n^{-1}$. That way, $I_n(e_{l_1} \otimes \ldots \otimes e_{l_n}) = f_{i_1} \wedge \ldots \wedge f_{i_p} \wedge g_{j_1} \wedge \ldots \wedge g_{j_m}$.

\begin{align*}
\tilde{\mu}(f_{i_1} \wedge \ldots \wedge f_{i_p} \wedge g_{j_1} \wedge \ldots \wedge g_{j_m}) &= I_n \circ (id_{W} \otimes \mu^{\otimes n-1})(e_{l_1} \otimes \ldots \otimes e_{l_n}) \\ 
 &={t_0^{-(n-1)}(-1)^{\# \{ k \in \{ 2, \ldots , n \} | l_k = 2 \} } (-1)^{\# \{ k \in \{ 2, \ldots , n \} | l_k = 3 \} }} \\
 & {I_n(e_{l_1} \otimes \ldots \otimes e_{l_n})} \\
 &=t_0^{-(n-1)}(-1)^{\# \{ k \in \{ 2, \ldots , n \} | l_k=2 \}} \\
 & \text{  } (-1)^{\# \{ k \in \{ 2, \ldots , n \} | l_k=3 \}} f_{i_1} \wedge \ldots \wedge f_{i_p} \wedge g_{j_1} \wedge \ldots \wedge g_{j_m}
\end{align*}  
\\
But :\\
$(-1)^{\# \{ k \in \{ 2, \ldots , n \} | l_k=2 \}} (-1)^{\# \{ k \in \{ 2, \ldots , n \} | l_k=3 \}} \\
= (-1)^{n-1} (-1)^{\# \{ k \in \{ 2, \ldots , n \} | l_k=1 \}} (-1)^{\# \{ k \in \{ 2, \ldots , n \} | l_k=4 \}} \\
= (-1)^{n-1} (-1)^{\# \{ k \in \{ 2, \ldots , n \} | l_k=1 \}} (-1)^{\# \{ k \in \{ 2, \ldots , n \} | l_k=4 \}} \left( (-1)^{\# \{ k \in \{ 2, \ldots , n \} | l_k=3 \}} \right)^2 \\
= (-1)^{n-1} (-1)^{\# \{ k \in \{ 2, \ldots , n \} | l_k=3 \text{ or } l_k=4 \}} (-1)^{\# \{ k \in \{ 2, \ldots , n \} | l_k=1 \text{ or } l_k=3 \}} \\
= (-1)^{n-1} (-1)^{\# \{ k \in \{ 2, \ldots , n \} | f_k \text{ appears} \}} (-1)^{\# \{ k \in \{ 2, \ldots , n \} | g_k \text{ appears} \}}
$
\\
\\
This provides the result.

\end{proof}

Given the expression for $\tilde{\mu}$ we just obtained, and the special form of representation $\Psi_n$ we have :

\begin{prop}
Invariant $LG(L ; t_0,t_0^{-1})$ can be written as a product, with each term depending only on one of the copies of the Burau representation. 
\end{prop}

\begin{proof}
Recall $LG(L ; t_0 , t_0^{-1}) = \frac{1}{4}~trace((id_{W} \otimes \mu^{\otimes n-1}) \circ b_R^n(b))$, where :\\
$$\mu = \begin{pmatrix}
   t_0^{-1} & . & . & . \\
   . & -t_0^{-1} & . & . \\
   . & . & -t_0^{-1} & . \\
   . & . & . & t_0^{-1}
\end{pmatrix}.$$ \\

Using that we can write :
\begin{flalign*}
LG(L ; t_0 , t_0^{-1}) &= \frac{1}{4}~trace( \tilde{\mu} \circ \Psi_n(b)) \\
&= \frac{1}{4}~ \sum_{\begin{array}{c} 
\scriptstyle 1 \leq i_1 < \ldots < i_p \leq n \\
\scriptstyle 1 \leq j_1 < \ldots < j_m \leq n
\end{array}} (f_{i_1} \wedge \ldots \wedge g_{j_m})^*\left( \tilde{\mu} \circ \Psi_n(b)(f_{i_1} \wedge \ldots \wedge g_{j_m}) \right) \\
\end{flalign*}
where $(f_{i_1} \wedge \ldots \wedge g_{j_m})^*$ indicates a vector of the dual basis of the reference basis. But given lemma \ref{mutilde},
\begin{align*}
(f_{i_1} \wedge \ldots \wedge g_{j_m})^*\left( \tilde{\mu} \circ \Psi_n(b)(f_{i_1} \wedge \ldots \wedge g_{j_m}) \right) 
&= (-t_0)^{-(n-1)} (-1)^{\# \{ k \in \{ 2, \ldots , n \} | f_k \text{ appears} \}} \\
& (-1)^{\# \{ k \in \{ 2, \ldots , n \} | g_k \text{ appears} \}} \\
&  (f_{i_1} \wedge \ldots \wedge g_{j_m})^*\left(\Psi_n(b)(f_{i_1} \wedge \ldots \wedge g_{j_m}) \right)
\end{align*}
\\
Also, $(f_{i_1} \wedge \ldots \wedge f_{i_p} \wedge g_{j_1} \wedge \ldots \wedge g_{j_m})^*(\underbrace{\Psi_n(b)(f_{i_1} \wedge \ldots \wedge f_{i_p} \wedge g_{j_1} \wedge \ldots \wedge g_{j_m}}_{\bigwedge F(b)(f_{i_1} \wedge \ldots \wedge f_{i_p})\wedge \bigwedge G(b)(g_{j_1} \wedge \ldots \wedge g_{j_m})}))$ \\
$\text{}$\ \ \ \ \ \ $= (f_{i_1} \wedge \ldots \wedge f_{i_p})^*(\bigwedge F(b)(f_{i_1} \wedge \ldots \wedge f_{i_p}))~(g_{j_1} \wedge \ldots \wedge g_{j_m})^*(\bigwedge G(b)(g_{j_1} \wedge \ldots \wedge g_{j_m}))$. \\
\\
That way we have the following expression for $LG(L ; t_0 , t_0^{-1})$:
\begin{align*}
\frac{1}{4}~(-t_0)^{-(n-1)} \sum_{\begin{array}{c} 
\scriptstyle 1 \leq i_1 < \ldots < i_p \leq n 
\end{array}} \Bigg( (-1)^{\# \{ k \in \{ 2, \ldots , n \} | f_k \text{ appears} \}} (f_{i_1} \wedge \ldots \wedge f_{i_p})^*(\bigwedge F(b)(f_{i_1} \wedge \ldots \wedge f_{i_p})) \Bigg) \\
* \sum_{\begin{array}{c} 
\scriptstyle 1 \leq j_1 < \ldots < j_m \leq n 
\end{array}} \Bigg( (-1)^{\# \{ k \in \{ 2, \ldots , n \} | g_k \text{ appears} \}} (g_{j_1} \wedge \ldots \wedge g_{j_m})^*(\bigwedge G(b)(g_{j_1} \wedge \ldots \wedge g_{j_m})) \Bigg)
\end{align*}
\end{proof}

Now we wish to show that each of these two sums is equal to $\Delta_L(t_0)$ up to multiplication by a unit of $\mathbb{C}[t_0^{\pm 1}]$, that is up to multiplication by $\pm t_0^n$, $n\in \mathbb{Z}$.
%
%
\section{Proof of the main theorem}

A careful analysis of \cite{Oht}, appendix C, shows that we have a coefficient in front of the partial trace in the expression of the Alexander-Conway polynomial :
$$
\Delta_L(t_0) \overset{\bullet}{=} t_0^{-(n-1)/2} ~ \frac{1}{2} ~ trace((id_{V} \otimes h^{\otimes n-1}) \circ \Psi_{V^{\otimes{n}}}(b)).
$$
So we can write more simply :
$$
\Delta_L(t_0) \overset{\bullet}{=}\frac{1}{2} ~ trace((id_{V} \otimes \tilde{h}^{\otimes n-1}) \circ \Psi_{V^{\otimes{n}}}(b)) \text{, where } \tilde{h} = \begin{pmatrix}
   1 & 0 \\
   0 & -1
\end{pmatrix}.
$$

\begin{prop}
Let $J_n$ : $V^{\otimes n}$ $\longrightarrow$ $\bigwedge W_n$, $e_{i_1} \otimes \ldots \otimes e_{i_n}$ $\longmapsto$ $\underset{k \text{ : } i_k = 1}{\bigwedge} f_k$. Then $J_n$ is a $\mathbb{C}[B_n]$-module automorphism :
$$
\bigwedge \Psi_{W_n}(b) \circ J_n = J_n \circ \Psi_{V^{\otimes n}}(b) \text{, } \forall b \in B_n .
$$
\end{prop}

The proof is quite similar to the one we did in the previous section. It is detailed in \cite{Oht}, appendix C, where what we just called $J_n$ is denoted by $I_n$, and is introduced by induction.

Applying $J_n$, we can express the Alexander polynomial differently : 
$$
\Delta_L(t_0) \overset{\bullet}{=}\frac{1}{2} ~ trace \big( \underbrace{J_n \circ (id_{V} \otimes \tilde{h}^{\otimes n-1}) \circ J_n^{-1}}_{\mu_1} \circ \bigwedge\Psi_{W_n}(b)\big).
$$
Where :
\begin{align*}
\mu_1(f_{i_1} \wedge \ldots \wedge f_{i_p}) &= J_n \circ (id_{V} \otimes \tilde{h}^{\otimes n-1})(e_0 \otimes \ldots \otimes e_0 \otimes \underbrace{e_1}_{i_1^{th} \text{ position}} \otimes e_0 \otimes \ldots \otimes \e_0 \otimes \underbrace{e_1}_{i_2^{th} \text{ position}} \otimes \ldots ) \\
&=(-1)^{\# \{ k \in \{ 2, \ldots , n \} | i_k = 1 \}} f_{i_1} \wedge \ldots \wedge f_{i_p} \\
&=(-1)^{\# \{ k \in \{ 2, \ldots , n \} | f_k \text{ appears} \}} f_{i_1} \wedge \ldots \wedge f_{i_p} .
\end{align*}\\
\\
Therefore, $$\Delta_L(t_0) \overset{\bullet}{=}\frac{1}{2} ~ \sum_{\begin{array}{c} 
\scriptstyle 1 \leq i_1 < \ldots < i_p \leq n 
\end{array}} \Bigg( (-1)^{\# \{ k \in \{ 2, \ldots , n \} | f_k \text{ appears} \}} (f_{i_1} \wedge \ldots \wedge f_{i_p})^*(\bigwedge \Psi_{W_n}(b)(f_{i_1} \wedge \ldots \wedge f_{i_p})) \Bigg).$$
But $F$ and $\Psi_{W_n}$ are identical once you change $t_0$ into $t_0^{-1}$. That way we can identify the first factor of our product with $\Delta_L(t_0^{-1})$. But the Alexander polynomial is symmetric : $\Delta_L(t_0)\overset{\bullet}{=}\Delta_L(t_0^{-1})$ \cite{ToFo}. So the only remaining problem is to identify the second sum with the Alexander invariant to be able to conclude. To do that we have to modify the representation of $V^{\otimes n}$ we used up to now to define $\Delta_L(t_0)$, and especially R-matrix $R_1$ we introduced at the beginning.

\begin{lem}
We can slightly modify R-matrix $R_1$ so that the new representations $\rho_{V^{\otimes{n}}}$ of the braid groups we obtain that way still verify :
$$
\Delta_L(t_0) \overset{\bullet}{=}\frac{1}{2} ~ trace((id_{V} \otimes \tilde{h}^{\otimes n-1}) \circ \rho_{V^{\otimes{n}}}(b)).
$$
\end{lem}

\begin{proof}
For the moment, we can write : $\Delta_L(t_0) \overset{\bullet}{=}\frac{1}{2} ~ trace((id_{V} \otimes \tilde{h}^{\otimes n-1}) \circ \Psi_{V^{\otimes{n}}}(b))$, where $\Psi_{V^{\otimes n}}$ is the representation associated to R-matrix $R_1 = \begin{pmatrix}
   1 & 0 & 0 & 0 \\
   0 & 0 & t_0^{1/2} & 0 \\
   0 & t_0^{1/2} & 1-t_0 & 0 \\
   0 & 0 & 0 & -t_0
\end{pmatrix}$.\\
\\
We can replace $R_1$ by $R_2 = -t_0^{-1} R_1 = \begin{pmatrix}
   -t_0^{-1} & 0 & 0 & 0 \\
   0 & 0 & -t_0^{-1/2} & 0 \\
   0 & -t_0^{-1/2} & 1-t_0^{-1} & 0 \\
   0 & 0 & 0 & 1
\end{pmatrix}$ in the definition of $\Psi_{V^{\otimes n}}$, and we will still have $\Delta_L(t_0) \overset{\bullet}{=}\frac{1}{2} ~ trace((id_{V} \otimes \tilde{h}^{\otimes n-1}) \circ \Psi_{V^{\otimes{n}}}(b))$. At last, we replace $t_0$ by $t_0^{-1}$ in $R_2$ to obtain $R_3 = \begin{pmatrix}
   -t_0 & 0 & 0 & 0 \\
   0 & 0 & -t_0^{1/2} & 0 \\
   0 & -t_0^{1/2} & 1-t_0 & 0 \\
   0 & 0 & 0 & 1
\end{pmatrix}$. We define the representation of $B_n$ associated with $R_3$ :
$$
\rho_{V^{\otimes n}}(\sigma_i) = id_V^{\otimes i-1} \otimes R_3 \otimes id_V^{\otimes n-i-1}.
$$
Since the Alexander polynomial is symmetric, we have the following expression for $\Delta_L(t_0)$, that will help us to conclude :
$$
\Delta_L(t_0) \overset{\bullet}{=}\frac{1}{2} ~ trace((id_{V} \otimes \tilde{h}^{\otimes n-1}) \circ \rho_{V^{\otimes{n}}}(b)).
$$
\end{proof}

Using the same strategy as previously, we wish to find $K_n$ : $V^{\otimes n}$ $\longrightarrow$ $\bigwedge W_n$ such that for any $b \in B_n$ : $\bigwedge G(b) \circ K_n = K_n \circ \rho_{V^{\otimes n}}(b)$.

\begin{prop}
We set $
K_1 = \left\{
    \begin{array}{lllll}
     V & \longrightarrow & \bigwedge W_1  \\
        e_1 & \longmapsto & 1 \\
        e_0 & \longmapsto & g_1 \\
    \end{array}
\right. 
$ and, for $n \geqslant 2$, $$
K_n = \left\{
    \begin{array}{cllll}
     V^{\otimes n} & \longrightarrow & \bigwedge W_n \\
        e_{i_1} \otimes \ldots \otimes e_{i_{n-1}} \otimes e_0 & \longmapsto & K_{n-1}(e_{i_1} \otimes \ldots \otimes e_{i_{n-1}}) \wedge g_n \\
        e_{i_1} \otimes \ldots \otimes e_{i_{n-1}} \otimes e_1 & \longmapsto & K_{n-1}(e_{i_1} \otimes \ldots \otimes e_{i_{n-1}}) \\
    \end{array}
\right. .
$$
Then, for any $b \in B_n$ : $\bigwedge G(b) \circ K_n = K_n \circ \rho_{V^{\otimes n}}(b)$.

\end{prop}

\begin{proof}
We leave it to the reader to verify that a proof by induction resembling the one we did with $I_n$ concludes.
\end{proof}

That way, \begin{align*}
\Delta_L(t_0) &\overset{\bullet}{=}\frac{1}{2} ~ trace((id_{V} \otimes \tilde{h}^{\otimes n-1}) \circ \rho_{V^{\otimes{n}}}(b)) \\
&= \frac{1}{2} ~ trace(\underbrace{K_n \circ (id_V \otimes \tilde{h}^{\otimes n-1}) \circ K_n^{-1}}_{\nu} \circ \bigwedge G(b)) .
\end{align*}

We can set $K_n(e_{i_1} \otimes \ldots \otimes e_{i_n})= g_{j_1} \wedge \ldots \wedge g_{j_m}$. That allows us to explicit the values of $\nu$ on the natural basis of $\bigwedge W_n$.
\begin{align*}
\nu(g_{j_1} \wedge \ldots \wedge g_{j_m})&= K_n \circ id_V \otimes \tilde{h}^{\otimes n-1} (e_{i_1} \otimes \ldots \otimes e_{i_n}) \\
&= (-1)^{\# \{ k \in \{ 2, \ldots , n \} | i_k = 1 \} } K_n(e_{i_1} \otimes \ldots \otimes e_{i_n}) \\
&= (-1)^{n-1} (-1)^{\# \{ k \in \{ 2, \ldots , n \} | i_k = 0 \} } g_{j_1} \wedge \ldots \wedge g_{j_m} \\
&= (-1)^{n-1} (-1)^{\# \{ k \in \{ 2, \ldots , n \} | g_k \text{ appears} \}} g_{j_1} \wedge \ldots \wedge g_{j_m} .
\end{align*}

So :
\begin{align*}
\Delta_L(t_0) &= \frac{1}{2} ~ trace(\nu \circ \bigwedge G(b)) \\
&\overset{\bullet}{=}\frac{1}{2} ~ \sum_{\begin{array}{c} 
\scriptstyle 1 \leq j_1 < \ldots < j_m \leq n 
\end{array}} \Bigg( (-1)^{\# \{ k \in \{ 2, \ldots , n \} | g_k \text{ appears} \}} (g_{j_1} \wedge \ldots \wedge g_{j_m})^*(\bigwedge G(b)(g_{j_1} \wedge \ldots \wedge g_{j_m})) \Bigg).
\end{align*}

$
\text{And finally } LG(L; t_0, t_0^{-1}) \overset{\bullet}{=} \Delta_L(t_0)^2
$ for any link $L$.

%
%
\section{Generalizing the proof}

\subsection{Writing the conjecture in case $(n,1)$ and other considerations}

The completely general conjecture states, using variables $(\tau, q)$ :
$$
LG^{m,n}( L ; \tau , e^{i \pi / n}) = \Delta_L(\tau^{2 n})^{m} \text{ , for any link } L \text{.}
$$
We can rewrite it using variables $(t_0,t_1)$. Indeed, since $q = e^{i \pi / n}$, the variables are related by $t_1^{1/2} = \tau^{-1} e^{i \pi / n}$ and $ t_0^{1/2}=\tau$. Therefore, the conjecture can be expressed the following way :
$$
LG^{m,n}( L ; t_0 , e^{2 i \pi / n} t_0^{-1}) = \Delta_L(t_0^n)^{m} \text{ , for any link } L \text{.}
$$
We can now explore in which cases it seems reasonable to attempt to generalize the strategy we used to evaluate the reduction of $LG^{2,1}$. An obvious obstruction to that concerns the dimension of both representations we built and showed they were isomorphic. Let's calculate the dimensions of the natural generalizations of these representations in case $(m,n)$. 

The vector space corresponding to what we denoted $W$ is the highest weight $U_q(gl(m|n))$-module used to define $LG^{m,n}$. It is $2^{nm}$-dimensional. So the representation of braid group $B_p$ defined thanks to the corresponding R-matrix is $2^{nmp}$-dimensional. On the other hand, the representation of $B_p$ we want to define to produce $\Delta_L(t_0^n)^{m}$ is :
$$
\bigwedge\underbrace{(W_p \oplus \ldots \oplus W_p)}_{m\text{ times}} 
$$
where each $W_p$ is a $\mathbb{C}[B_p]$-module isomorphic to a version of $\Psi_{W_p}$ where $t_0$ is replaced by $t_0^{\pm n}$. Such a representation is $2^{mp}$-dimensional. These two representations can not be isomorphic if $n>1$.

That is why a straightforward use of our method can only be applied to prove cases $(m,1)$. 

\subsection{Proof of case $(3,1)$}

We give the essential steps to prove the result that interests us in the case $(m,n) = (3,1)$. We follow the same ideas we used to study $LG^{2,1}$.

\begin{thm}
For any oriented link L, $LG^{3,1}(L; t_0, t_0^{-1}) \overset{\bullet}{=} \Delta_L(t_0)^3$.
\end{thm}

For an explicit definition of $LG^{3,1}$, see \cite{DeWit}, p.17. The author uses variables $(\tau, q)$, but denotes $\tau = q^{- \alpha}$. We will only use the reduced version of $LG^{3,1}$. It is obtained by setting $q=-1$ and $q^{- \alpha} = t_0^{1/2}$.

\begin{rem}
Since we are going to set $q=-1$ in the R-matrix of \cite{DeWit}, we have to chose precisely what the roots that are written formally are. We have chosen : $[\alpha +1]^{1/2} = q^{-1/2} [\alpha]^{1/2}$ and $[\alpha +2]^{1/2} = - [\alpha]^{1/2}$.
\end{rem}

\begin{defn}
(\textit{R-matrix $S$}) \\
Set $\mathbb{F} := \mathbb{C}(t_0^{\pm 1/2})$. Let $W = <e_1, \ldots, e_8>$ be a $8$-dimensional $\mathbb{F}$-vector space. We define $S$ an automorphism of $W \otimes W$ as the direct sum of the following automorphisms ($S$ is globally multiplied by $t_0^{-3/2}$ in comparison with the R-matrix explicited in \cite{DeWit}) :
$$\begin{pmatrix}
1& . & . & . & . & . & . & . \\
. & -t_0^{-1} & . & . & . & . & . & . \\
. & . & -t_0^{-1} & . & . & . & . & . \\
. & . & . & -t_0^{-1} & . & . & . & . \\
. & . & . & . & t_0^{-2} & . & . & . \\
. & . & . & . & . & t_0^{-2} & . & . \\
. & . & . & . & . & . & t_0^{-2} & . \\
. & . & . & . & . & . & . & -t_0^{-3} \\
\end{pmatrix}
$$
in basis $(e_1 \otimes e_1,e_2 \otimes e_2,e_3 \otimes e_3,e_4 \otimes e_4,e_5 \otimes e_5,e_6 \otimes e_6,e_7 \otimes e_7,e_8 \otimes e_8 )$ ;\\
\\
several copies of 
$$
\begin{pmatrix}
0 & t_0^{-1/2} \\
t_0^{-1/2} & 1-t_0^{-1}\\
\end{pmatrix}
$$
in bases $(e_1 \otimes e_2 , e_2 \otimes e_1)$, $(e_1 \otimes e_3 , e_3 \otimes e_1)$ and $(e_1 \otimes e_4 , e_4 \otimes e_1)$ ; \\
\\
several copies of
$$ t_0^{-2}
\begin{pmatrix}
0 & t_0^{-1/2} \\
t_0^{-1/2} & 1-t_0^{-1}\\
\end{pmatrix}
$$
in bases $(e_7 \otimes e_8 , e_8 \otimes e_7)$, $(e_6 \otimes e_8 , e_8 \otimes e_6)$ and $(e_5 \otimes e_8 , e_8 \otimes e_5)$ ;\\
\\
several copies of
$$ -t_0^{-1}
\begin{pmatrix}
0 & t_0^{-1/2} \\
t_0^{-1/2} & 1-t_0^{-1}\\
\end{pmatrix}
$$
in bases $(e_2 \otimes e_5 , e_5 \otimes e_2)$, $(e_3 \otimes e_5 , e_5 \otimes e_3)$, $(e_2 \otimes e_6 , e_6 \otimes e_2)$, $(e_4 \otimes e_6 , e_6 \otimes e_4)$, $(e_3 \otimes e_7 , e_7 \otimes e_3)$ and $(e_4 \otimes e_7 , e_7 \otimes e_4)$ ; \\
\\
several copies of 
$$ t_0^{-1}
\begin{pmatrix}
. & . & . & 1 \\
. & . & 1 & t_0^{1/2}-t_0^{-1/2} \\ 
. & 1 & . & t_0^{1/2}-t_0^{-1/2} \\
1 & t_0^{1/2}-t_0^{-1/2} & t_0^{1/2}-t_0^{-1/2} & (t_0^{1/2}-t_0^{-1/2})^2 \\
\end{pmatrix}
$$
in bases $(e_1 \otimes e_5,e_2 \otimes e_3, e_3 \otimes e_2, e_5 \otimes e_1)$, $(e_1 \otimes e_6,e_2 \otimes e_4, e_4 \otimes e_2, e_6 \otimes e_1)$ and $(e_1 \otimes e_7,e_3 \otimes e_4, e_4 \otimes e_3, e_7 \otimes e_1)$ ; \\
\\
several copies of 
$$ -t_0^{-2}
\begin{pmatrix}
. & . & . & 1 \\
. & . & 1 & t_0^{1/2}-t_0^{-1/2} \\ 
. & 1 & . & t_0^{1/2}-t_0^{-1/2} \\
1 & t_0^{1/2}-t_0^{-1/2} & t_0^{1/2}-t_0^{-1/2} & (t_0^{1/2}-t_0^{-1/2})^2 \\
\end{pmatrix}
$$
in bases $(e_4 \otimes e_8,e_6 \otimes e_7, e_7 \otimes e_6, e_8 \otimes e_4)$, $(e_3 \otimes e_8,e_5 \otimes e_7, e_7 \otimes e_5, e_8 \otimes e_3)$ and $(e_2 \otimes e_8,e_5 \otimes e_6, e_6 \otimes e_5, e_8 \otimes e_2)$ ; \\
\\
\footnotesize
$$ t_0^{-3/2}
\begin{pmatrix}
. & . & . & . & . & . & . & 1 \\
. & . & . & . & . & . & 1 & t_0^{1/2}-t_0^{-1/2} \\
. & . & . & . & . & 1& . & t_0^{1/2}-t_0^{-1/2}\\
. & . & . & . & 1& . & . & t_0^{1/2}-t_0^{-1/2}\\
. & . & . & 1 & . & t_0^{1/2}-t_0^{-1/2}& t_0^{1/2}-t_0^{-1/2}& (t_0^{1/2}-t_0^{-1/2})^2\\
. & . & 1& . & t_0^{1/2}-t_0^{-1/2}& . & t_0^{1/2}-t_0^{-1/2}& (t_0^{1/2}-t_0^{-1/2})^2\\
. & 1 & . & . & t_0^{1/2}-t_0^{-1/2}& t_0^{1/2}-t_0^{-1/2}& . & (t_0^{1/2}-t_0^{-1/2})^2\\
 1&t_0^{1/2}-t_0^{-1/2} &t_0^{1/2}-t_0^{-1/2} &t_0^{1/2}-t_0^{-1/2} &(t_0^{1/2}-t_0^{-1/2})^2 &(t_0^{1/2}-t_0^{-1/2})^2 &(t_0^{1/2}-t_0^{-1/2})^2 & (t_0^{1/2}-t_0^{-1/2})^3\\ 
\end{pmatrix}$$
\normalsize
in basis $(e_1 \otimes e_8,e_4 \otimes e_5, e_3 \otimes e_6, e_2 \otimes e_7,e_7 \otimes e_2,e_6 \otimes e_3, e_5 \otimes e_4, e_8 \otimes e_1)$. \\
\\
Then $S$ is an R-matrix. So we can denote by $b_S^n$ the representation of braid group $B_n$ derived from $S$. It is given by the usual expression :
$$
b_S^n(\sigma_i) = id_{W}^{\otimes{i-1}} \otimes S \otimes id_{W}^{\otimes{n-i-1}} \text{ , } i = 1, \ldots, n-1 .
$$
\end{defn}

\begin{defn}
\textit{Reduced Links-Gould invariant $LG^{3,1}$} \\
Let $L$ be any oriented link, and $b \in B_n$ be a braid with closure $L$. The reduced version of Links-Gould invariant $LG^{3,1}$ is given by the following formula :
$$
LG^{3,1}(L ; t_0 , t_0^{-1}) = \frac{1}{8} ~ trace( (id_W \otimes \mu^{\otimes n-1}) \circ b_S^n(b))
$$
where $$\mu = t_0^{3/2} \begin{pmatrix}
 1 & . & . & . & . & . & . & . \\
 . & -1 & . & . & . & . & . & . \\
 . & . & -1 & . & . & . & . & . \\
 . &.  &.  & -1 & . & . & . & . \\
 . & . & . & . & 1 & . & . & . \\
 . & . & . & . & . & 1 & . & . \\
 . & . & . & . & . & . & 1 & . \\
 . & . & . & . & . & . & . & -1 \\ 
\end{pmatrix}~ \in End(W).$$
\end{defn}

We set three $n$-dimensional vector spaces $<f_1, \ldots , f_n>$, $<g_1, \ldots , g_n>$ and $<h_1, \ldots , h_n>$ that will be all refered to as $W_n$. On each of them, we define a representation isomorphic to the Burau representation :
$$
F(\sigma_i)(f_j) = \left\{
    \begin{array}{lll}
     t_0^{-1/2} f_{i+1} & \mbox{if } j=i ~, \\
        t_0^{-1/2} f_i + (1-t_0^{-1}) f_{i+1}& \mbox{if }j=i+1 ~, \\
        f_j & \mbox{otherwise}.
    \end{array}
\right.
$$
We designate by $G$ and $H$ representations on $<g_1, \ldots , g_n>$ and $<h_1, \ldots , h_n>$ defined by the exact same formula. Then we set $\Phi_n$ the representation of $B_n$ on $\bigwedge (W_n \oplus W_n \oplus W_n)$ given by :
$$
\Phi_n := \bigwedge(F \oplus G \oplus H).
$$
When $n=2$, one can compute $\Phi_2(\sigma_1)$ and notice that its matrix is equal to $S$ in a well chosen basis. A precise look at this basis gave us the idea to define the following map by induction. Note that retrospectively one can recover this basis simply by computing the image by our map of the basis we used to express $S$ when $n=2$. 
\begin{thm}
We set $
I_1 = \left\{
    \begin{array}{lllll}
     V & \longrightarrow & \bigwedge W_1  \\
        e_1 & \longmapsto & 1 \\
        e_2 & \longmapsto & f_1 \\
        e_3 & \longmapsto & g_1 \\
        e_4 & \longmapsto & h_1 \\
        e_5 & \longmapsto & f_1 \wedge g_1 \\
        e_6 & \longmapsto & f_1 \wedge h_1 \\
        e_7 & \longmapsto & g_1 \wedge h_1 \\
        e_8 & \longmapsto & f_1 \wedge g_1 \wedge h_1 \\
    \end{array}
\right. 
$ and, for $n \geqslant 2$, $$
I_n = \left\{
    \begin{array}{cllll}
     V^{\otimes n} & \longrightarrow & \bigwedge W_n \\
        e_1 \otimes e_{i_{n-1}} \otimes \ldots \otimes e_{i_1} & \longmapsto & I_{n-1}(e_{i_{n-1}} \otimes \ldots \otimes e_{i_1}) \\
        e_2 \otimes e_{i_{n-1}} \otimes \ldots \otimes e_{i_1} & \longmapsto & Reord(I_{n-1}(e_{i_{n-1}} \otimes \ldots \otimes e_{i_1}) \wedge f_n) \\
        e_3 \otimes e_{i_{n-1}} \otimes \ldots \otimes e_{i_1} & \longmapsto & Reord(I_{n-1}(e_{i_{n-1}} \otimes \ldots \otimes e_{i_1}) \wedge g_n) \\
        e_4 \otimes e_{i_{n-1}} \otimes \ldots \otimes e_{i_1} & \longmapsto & I_{n-1}(e_{i_{n-1}} \otimes \ldots \otimes e_{i_1}) \wedge h_n \\
        e_5 \otimes e_{i_{n-1}} \otimes \ldots \otimes e_{i_1} & \longmapsto & Reord(I_{n-1}(e_{i_{n-1}} \otimes \ldots \otimes e_{i_1}) \wedge f_n \wedge g_n) \\
        e_6 \otimes e_{i_{n-1}} \otimes \ldots \otimes e_{i_1} & \longmapsto & Reord(I_{n-1}(e_{i_{n-1}} \otimes \ldots \otimes e_{i_1}) \wedge f_n \wedge h_n) \\
        e_7 \otimes e_{i_{n-1}} \otimes \ldots \otimes e_{i_1} & \longmapsto & Reord(I_{n-1}(e_{i_{n-1}} \otimes \ldots \otimes e_{i_1}) \wedge g_n \wedge h_n) \\
        e_8 \otimes e_{i_{n-1}} \otimes \ldots \otimes e_{i_1} & \longmapsto & Reord(I_{n-1}(e_{i_{n-1}} \otimes \ldots \otimes e_{i_1}) \wedge f_n \wedge g_n \wedge h_n) \\
    \end{array}
\right. .
$$
Then the following identity holds for $n \geqslant 1$ and $b \in B_n$ :
$$
\Phi_n(b) \circ I_n = I_n \circ b_S^n(\hat{b}) .
$$
\end{thm}

\begin{rem}
As in the previous sections, $Reord$ refers to a reference basis of $\bigwedge (W_n \oplus W_n \oplus W_n)$ that is :$$
(f_{i_1} \wedge \ldots \wedge f_{i_p} \wedge g_{j_1} \wedge \ldots \wedge g_{j_m} \wedge h_{k_1} \wedge \ldots \wedge h_{k_q})_{1 \leq i_1 < \ldots < i_p \leq n \text{, } 
1 \leq j_1 < \ldots < j_m \leq n \text{, } 
1 \leq k_1 < \ldots < k_q \leq n}
$$
\end{rem}

\begin{rem}
For $b = \sigma_{i_1}^{\varepsilon_1} \ldots \sigma_{i_p}^{\varepsilon_p} \in B_n$, we define $\hat{b} := \sigma_{n-i_1}^{\varepsilon_1} \ldots \sigma_{n-i_p}^{\varepsilon_p}$. $\hat{b}$ is braid $b$ "looked at from the other side". That way we have elementary properties : $closure(b) = closure(\hat{b})$ ; $\hat{\sigma_k} = \sigma_{n-k}$ ; for any $\sigma$, $\tau$ $\in B_n$ : $\hat{\sigma \tau} = \hat{\sigma} \hat{\tau}$.
\end{rem}

We can use $I_n$ to express $LG^{3,1}$ differently.

\begin{align*}
LG^{3,1}(L;t_0,t_0^{-1}) &= \frac{1}{8} ~ trace( (id_W \otimes \mu^{\otimes n-1}) \circ b_S^n(b)) \\
&= \frac{1}{8} ~ trace(\underbrace{I_n \circ (id_W \otimes \mu^{\otimes n-1}) \circ I_n^{-1}}_{\tilde{\mu}} \circ  \Phi_n(\hat{b})) .
\end{align*}

Denoting as we already did several times  $I_n(e_{i_n} \otimes \ldots \otimes e_{i_1}) = f_{i_1} \wedge \ldots \wedge h_{k_q}$, we can compute $\tilde{\mu}$ :

\begin{align*}
\tilde{\mu}(f_{i_1} \wedge \ldots \wedge h_{k_q}) &= I_n \circ (id_W \otimes \mu^{\otimes n-1})(e_{i_n} \otimes \ldots \otimes e_{i_1}) \\
&= t_0^{3(n-1)/2} (-1)^{ \# \{k \in \{ 1, \ldots , n-1 \} | i_k \in \{ 2,3,4,8 \} \}} I_n(e_{i_n} \otimes \ldots \otimes e_{i_1}) \\
&= t_0^{3(n-1)/2} (-1)^{ \# \{k \in \{ 1, \ldots , n-1 \} | \text{ an odd number of the following appear : } \{ f_k,g_k,h_k\} \}} \\
& \ \ \ \ f_{i_1} \wedge \ldots \wedge h_{k_q} \\
&= t_0^{3(n-1)/2} (-1)^{ \# \{k \in \{ 1, \ldots , n-1 \} | f_k \text{ appears}  \}} (-1)^{ \# \{k \in \{ 1, \ldots , n-1 \} | g_k \text{ appears}  \}} \\
& \ \ \ \ (-1)^{ \# \{k \in \{ 1, \ldots , n-1 \} | h_k \text{ appears}  \}} f_{i_1} \wedge \ldots \wedge h_{k_q} .
\end{align*}
\\
So 
\begin{align*}
LG^{3,1}(L;t_0,t_0^{-1}) &= \frac{1}{8} ~ trace(\tilde{\mu} \circ \Phi_n(\hat{b})) \\
&= \frac{1}{8} \sum_{\begin{array}{c} 
\scriptstyle 1 \leq i_1 < \ldots < i_p \leq n \\
\scriptstyle 1 \leq j_1 < \ldots < j_m \leq n \\
\scriptstyle 1 \leq k_1 < \ldots < k_q \leq n
\end{array}} (f_{i_1} \wedge \ldots \wedge h_{k_q})^{*}(\tilde{\mu} \circ \Phi_n(\hat{b})(f_{i_1} \wedge \ldots \wedge h_{k_q})) 
\end{align*}
But
\begin{align*}
&(f_{i_1} \wedge \ldots \wedge h_{k_q})^{*}(\tilde{\mu} \circ \Phi_n(\hat{b})(f_{i_1} \wedge \ldots \wedge h_{k_q})) \\
&= t_0^{3(n-1)/2} (-1)^{ \# \{k \in \{ 1, \ldots , n-1 \} | f_k \text{ appears}  \}} (-1)^{ \# \{k \in \{ 1, \ldots , n-1 \} | g_k \text{ appears}  \}} \\
& \ \ \ \ (-1)^{ \# \{k \in \{ 1, \ldots , n-1 \} | h_k \text{ appears}  \}} (f_{i_1} \wedge \ldots \wedge h_{k_q})^{*}(\Phi_n(\hat{b})(f_{i_1} \wedge \ldots \wedge h_{k_q})) \\
&= t_0^{3(n-1)/2} (-1)^{\ldots} (f_{i_1} \wedge \ldots \wedge f_{i_p})^{*}(\bigwedge F(\hat{b})(f_{i_1} \wedge \ldots \wedge f_{i_p})) \\
& \ \ \ \ (g_{j_1} \wedge \ldots \wedge g_{j_m})^{*}(\bigwedge G(\hat{b})(g_{j_1} \wedge \ldots \wedge g_{j_m})) ~ (h_{k_1} \wedge \ldots \wedge h_{k_q})^{*}(\bigwedge H(\hat{b})(h_{k_1} \wedge \ldots \wedge h_{k_q}))
\end{align*}
So finally 
\begin{align*}
LG^{3,1}(L;t_0,t_0^{-1})&= t_0^{3(n-1)/2} \Bigg( \frac{1}{2} \sum_{\begin{array}{c} 
\scriptstyle 1 \leq i_1 < \ldots < i_p \leq n \\
\end{array}} (-1)^{ \# \{k \in \{ 1, \ldots , n-1 \} | f_k \text{ appears}  \}} \\
& \ \ \ \ \ \ \ \ (f_{i_1} \wedge \ldots \wedge f_{i_p})^{*}(\bigwedge F(\hat{b})(f_{i_1} \wedge \ldots \wedge f_{i_p})) \Bigg) * \dots
\end{align*}

The only thing that remains to be shown is that each of the three terms in the product is equal to $\Delta_{\hat{b}}(t) = \Delta_b(t)$. The proof is similar to the one we did for $LG^{2,1}$. The main point is to find R-matrices associated to their representations on $V^{\otimes n}$ such that 
$$
\Delta_L(t) = \frac{1}{2}~trace((id_{V} \otimes h^{\otimes n-1}) \circ \Psi_{V^{\otimes{n}}}(b))
$$ 
and that up to conjugation the trace is one of the three sums. We will not detail this argument.

\subsection{A remark around case $(n,1)$}

To prove the identity
$$
LG^{n,1}(L;t_0,t_0^{-1}) \overset{\bullet}{=} \Delta_L(t_0)^n
$$
when $n = 2$, $3$, we have used the crucial fact that we know an explicit formula for the R-matrix and the left handle (the maps we called $\mu$) in these two cases. Solving the conjecture for any $n$ using the same ideas  therefore requires the R-matrix to be computed in all cases. In \cite{DeWit}, the calculations are explicit up to $n=4$. However, we believe it is possible, with a proper amount of sweat and will, to give a formula for any $n$. Indeed, $U_q(gl(n|1))$ is a quantum super-algebra and one can find expressions for universal R-matrices in that context in \cite{KhTo} or \cite{Yam}. More recently, M.D. Gould, P.S. Isaac and J.L. Werry wrote the representation that derives the R-matrix from it's universal counterpart in a practical basis \cite{Gou}. This allows to project the universal R-matrix to find the corresponding map.

\section{Appendix : proof of proposition \ref{commutation}} \label{appendix}

Here we prove the result we stated in proposition \ref{commutation}. That is : for any $b\in B_n$,
$$
\Psi_n(b) \circ I_n = I_n \circ b_R^n(b).
$$

\begin{proof}

We show the commutation by induction on $n$, the number of strands in the braid group we consider. Note that it has already been verified when $n = 1$, $2$. Let us now suppose the equality holds for $n-1$, $n \geqslant 3$. We only need to prove the result for $b = \sigma_k$, $k = 1 , \ldots , n-1$. \\
\\
\underline{For $\sigma_k$, $k \leq n-2$} : \\
\\
$I_n(b_R^n(\sigma_k)(e_{i_1} \otimes \ldots \otimes e_{i_n})) \\
= I_n(e_{i_1} \otimes \ldots \otimes R(e_{i_k} \otimes e_{i_{k+1}}) \otimes \ldots \otimes e_{i_n}) \\
= I_n(b_R^{n-1}(\sigma_k)(e_{i_1} \otimes \ldots \otimes e_{i_{n-1}}) \otimes e_{i_n}) \\
\\
= \left\{
    \begin{array}{ll}
        I_{n-1}(b_R^{n-1}(\sigma_k)(e_{i_1} \otimes \ldots \otimes e_{i_{n-1}})) \wedge g_n & \mbox{if } i_n = 1 \\
        I_{n-1}(b_R^{n-1}(\sigma_k)(e_{i_1} \otimes \ldots \otimes e_{i_{n-1}})) & \mbox{if } i_n = 2 \\
        Reord(I_{n-1}(b_R^{n-1}(\sigma_k)(e_{i_1} \otimes \ldots \otimes e_{i_{n-1}})) \wedge f_n \wedge g_n ) & \mbox{if } i_n = 3 \\
        Reord(I_{n-1}(b_R^{n-1}(\sigma_k)(e_{i_1} \otimes \ldots \otimes e_{i_{n-1}})) \wedge f_n ) & \mbox{if } i_n = 4
    \end{array}
\right. \\
= \left\{
    \begin{array}{ll}
        \Psi_{n-1}(\sigma_k)(I_{n-1}(e_{i_1} \otimes \ldots \otimes e_{i_{n-1}})) \wedge g_n & \mbox{if } i_n = 1 \\
        \Psi_{n-1}(\sigma_k)(I_{n-1}(e_{i_1} \otimes \ldots \otimes e_{i_{n-1}})) & \mbox{if } i_n = 2 \\
        Reord(I_{n-1}(b_R^{n-1}(\sigma_k)(e_{i_1} \otimes \ldots \otimes e_{i_{n-1}})) \wedge f_n \wedge g_n ) & \mbox{if } i_n = 3 \\
        Reord(I_{n-1}(b_R^{n-1}(\sigma_k)(e_{i_1} \otimes \ldots \otimes e_{i_{n-1}})) \wedge f_n ) & \mbox{if } i_n = 4
    \end{array}
\right. \text{ (inductive hypothesis)} \\
\\
$
\\
On the other hand :\\
$
\\
\Psi_{n}(\sigma_k)(I_{n}(e_{i_1} \otimes \ldots \otimes e_{i_n})) \\
= \left\{
    \begin{array}{ll}
        \Psi_{n}(\sigma_k)(I_{n-1}(e_{i_1} \otimes \ldots \otimes e_{i_{n-1}}) \wedge g_n) & \mbox{if } i_n = 1 \\
        \Psi_{n}(\sigma_k)(I_{n-1}(e_{i_1} \otimes \ldots \otimes e_{i_{n-1}})) & \mbox{if } i_n = 2 \\
        \Psi_{n}(\sigma_k)(Reord(I_{n-1}(e_{i_1} \otimes \ldots \otimes e_{i_{n-1}}) \wedge f_n \wedge g_n)) & \mbox{if } i_n = 3 \\
        \Psi_{n}(\sigma_k)(Reord(I_{n-1}(e_{i_1} \otimes \ldots \otimes e_{i_{n-1}}) \wedge f_n)) & \mbox{if } i_n = 4
    \end{array}
\right. \\
\\
\\
= \left\{
    \begin{array}{lll}
        \Psi_{n}(\sigma_k)(I_{n-1}(e_{i_1} \otimes \ldots \otimes e_{i_{n-1}})) \wedge \Psi_{n}(\sigma_k)(g_n) & \mbox{if } i_n = 1 \\
        \Psi_{n}(\sigma_k)(I_{n-1}(e_{i_1} \otimes \ldots \otimes e_{i_{n-1}})) = \Psi_{n-1}(\sigma_k)(I_{n-1}(e_{i_1} \otimes \ldots \otimes e_{i_{n-1}})) & \mbox{if } i_n = 2 \\
        \Psi_{n}(\sigma_k)(Reord(I_{n-1}(e_{i_1} \otimes \ldots \otimes e_{i_{n-1}}) \wedge f_n \wedge g_n)) & \mbox{if } i_n = 3 \\
        \Psi_{n}(\sigma_k)(Reord(I_{n-1}(e_{i_1} \otimes \ldots \otimes e_{i_{n-1}}) \wedge f_n)) & \mbox{if } i_n = 4 
    \end{array}
\right. 
$ \\
\\
\\
For $i_n=1$, since $\Psi_{n}(\sigma_k)(g_n)= g_n$, we obtain the result. We also observe the equality holds when $i_n=2$. We now study the two remaining cases.\\
\\
Let $\mu(e_{i_1} \otimes \ldots \otimes e_{i_n})$ be the total of the number of $e_1$ and the number of $e_3$ in that elementary tensor. Given the expression of $I_n$ and our reference basis of $\bigwedge(W_n \oplus W_n)$, it is obvious that : 
$$
Reord(I_{n-1}(e_{i_1} \otimes \ldots \otimes e_{i_{n-1}}) \wedge f_n) = (-1)^{\mu(e_{i_1} \otimes \ldots \otimes e_{i_{n-1}})} I_{n-1}(e_{i_1} \otimes \ldots \otimes e_{i_{n-1}}) \wedge f_n $$
and
$$Reord(I_{n-1}(e_{i_1} \otimes \ldots \otimes e_{i_{n-1}}) \wedge f_n \wedge g_n) = (-1)^{\mu(e_{i_1} \otimes \ldots \otimes e_{i_{n-1}})} I_{n-1}(e_{i_1} \otimes \ldots \otimes e_{i_{n-1}}) \wedge f_n \wedge g_n .
$$\\
Therefore : \\
\\
$\Psi_{n}(\sigma_k)(Reord(I_{n-1}(e_{i_1} \otimes \ldots \otimes e_{i_{n-1}}) \wedge f_n \wedge g_n))\\ = (-1)^{\mu(e_{i_1} \otimes \ldots \otimes e_{i_{n-1}})} \Psi_{n}(\sigma_k)(I_{n-1}(e_{i_1} \otimes \ldots \otimes e_{i_{n-1}}) \wedge f_n \wedge g_n)\\ = (-1)^{\mu(e_{i_1} \otimes \ldots \otimes e_{i_{n-1}})} \Psi_{n-1}(\sigma_k)(I_{n-1}(e_{i_1} \otimes \ldots \otimes e_{i_{n-1}})) \wedge f_n \wedge g_n
$ \\
\\
Moreover, given the specific form of matrix $R$, every term that appears in $b_R^{n-1}(\sigma_k)(e_{i_1} \otimes \ldots \otimes e_{i_{n-1}})$ has the same total number of $e_1$ and $e_3$ as $e_{i_1} \otimes \ldots \otimes e_{i_{n-1}}$. Hence : \\
\\
$Reord(I_{n-1}(b_R^{n-1}(\sigma_k)(e_{i_1} \otimes \ldots \otimes e_{i_{n-1}})) \wedge f_n \wedge g_n )\\ = (-1)^{\mu(e_{i_1} \otimes \ldots \otimes e_{i_{n-1}})} I_{n-1} \circ b_R^{n-1}(\sigma_k)(e_{i_1} \otimes \ldots \otimes e_{i_{n-1}})\wedge f_n \wedge g_n \\ =  (-1)^{\mu(e_{i_1} \otimes \ldots \otimes e_{i_{n-1}})} \Psi_{n-1}(\sigma_k) \circ I_{n-1}(e_{i_1} \otimes \ldots \otimes e_{i_{n-1}})\wedge f_n \wedge g_n
$ \\
\\
Thus we obtain the identity in case $i_n = 3$. Similar calculations show it is also true when $i_n = 4$. The only remaining question is for the last generator of $B_n$.\\
\\
\underline{For $\sigma_{n-1}$} : We show that $\Psi_n(\sigma_{n-1}) \circ I_n (e_{i_1} \otimes \ldots \otimes e_{i_{n}}) = I_n \circ b_R^n(\sigma_{n-1})(e_{i_1} \otimes \ldots \otimes e_{i_{n}})$ for each of the 16 possible ordered pairs $(i_{n-1} , i_n)$ : \\
\\
\boxed{$(1,1)$} : \\
$\Psi_n(\sigma_{n-1})(I_{n}(e_{i_1} \otimes \ldots \otimes e_1 \otimes e_1)) \\ = \Psi_n(\sigma_{n-1})(I_{n-2}(e_{i_1} \otimes \ldots \otimes e_{i_{n-2}}) \wedge g_{n-1} \wedge g_n ) \\
= I_{n-2}(e_{i_1} \otimes \ldots \otimes e_{i_{n-2}}) \wedge (-t_0^{1/2} g_n)\wedge (-t_0^{1/2} g_{n-1} + (1-t_0) g_n) \\
= -t_0 I_{n-2}(e_{i_1} \otimes \ldots \otimes e_{i_{n-2}}) \wedge g_{n-1} \wedge g_n \\
\\
I_n \circ (1 \otimes 1 \otimes \ldots \otimes R)(e_{i_1} \otimes \ldots \otimes e_1 \otimes e_1) \\ = I_n(e_{i_1} \otimes \ldots \otimes e_{i_{n-2}} \otimes -t_0 e_1 \otimes e_1) \\
= -t_0 I_{n-2}(e_{i_1} \otimes \ldots \otimes e_{i_{n-2}}) \wedge g_{n-1} \wedge g_n
$ \\
\\
Now that we have explicited one case, we give the results for the remaining ones. \\
\\
\boxed{$(4,4)$} :  $\Psi_n(\sigma_{n-1})(I_{n}(e_{i_1} \otimes \ldots \otimes e_4 \otimes e_4)) = I_n \circ (1 \otimes 1 \otimes \ldots \otimes R)(e_{i_1} \otimes \ldots \otimes e_4 \otimes e_4) \\ = -t_0^{-1} I_{n-2}(e_{i_1} \otimes \ldots \otimes e_{i_{n-2}}) \wedge f_{n-1} \wedge f_n $\\
\\
\boxed{$(2,2)$} :  $\Psi_n(\sigma_{n-1})(I_{n}(e_{i_1} \otimes \ldots \otimes e_2 \otimes e_2)) = I_n \circ (1 \otimes 1 \otimes \ldots \otimes R)(e_{i_1} \otimes \ldots \otimes e_2 \otimes e_2) \\= I_{n-2}(e_{i_1} \otimes \ldots \otimes e_{i_{n-2}})
$ \\
\\
\boxed{$(3,3)$} :  $\Psi_n(\sigma_{n-1})(I_{n}(e_{i_1} \otimes \ldots \otimes e_3 \otimes e_3)) = I_n \circ (1 \otimes 1 \otimes \ldots \otimes R)(e_{i_1} \otimes \ldots \otimes e_3 \otimes e_3) \\ = I_{n-2}(e_{i_1} \otimes \ldots \otimes e_{i_{n-2}}) \wedge f_{n-1} \wedge f_n \wedge g_{n-1} \wedge g_n
$ \\
\\
\boxed{$(1,2)$} :  $\Psi_n(\sigma_{n-1})(I_{n}(e_{i_1} \otimes \ldots \otimes e_1 \otimes e_2)) = I_n \circ (1 \otimes 1 \otimes \ldots \otimes R)(e_{i_1} \otimes \ldots \otimes e_1 \otimes e_2) \\ = -t_0^{1/2} I_{n-2}(e_{i_1} \otimes \ldots \otimes e_{i_{n-2}}) \wedge g_n
$ \\
\\
\boxed{$(2,1)$} :  $\Psi_n(\sigma_{n-1})(I_{n}(e_{i_1} \otimes \ldots \otimes e_2 \otimes e_1)) = I_n \circ (1 \otimes 1 \otimes \ldots \otimes R)(e_{i_1} \otimes \ldots \otimes e_2 \otimes e_1) \\ = I_{n-2}(e_{i_1} \otimes \ldots \otimes e_{i_{n-2}}) \wedge (-t_0^{1/2} g_{n-1} + (1-t_0) g_n)
$
\\
\\
\boxed{$(1,3)$} :  $\Psi_n(\sigma_{n-1})(I_{n}(e_{i_1} \otimes \ldots \otimes e_1 \otimes e_3)) = I_n \circ (1 \otimes 1 \otimes \ldots \otimes R)(e_{i_1} \otimes \ldots \otimes e_1 \otimes e_3) \\ = (-1)^{\mu(e_{i_1} \otimes \ldots \otimes e_{i_{n-2}} \otimes e_1)} t_0^{1/2} I_{n-2}(e_{i_1} \otimes \ldots \otimes e_{i_{n-2}}) \wedge f_{n-1} \wedge g_{n-1} \wedge g_n
$
\\
\\\boxed{$(3,1)$} :  $\Psi_n(\sigma_{n-1})(I_{n}(e_{i_1} \otimes \ldots \otimes e_3 \otimes e_1)) = I_n \circ (1 \otimes 1 \otimes \ldots \otimes R)(e_{i_1} \otimes \ldots \otimes e_3 \otimes e_1) \\ = (-1)^{\mu(e_{i_1} \otimes \ldots \otimes e_{i_{n-2}})} I_{n-2}(e_{i_1} \otimes \ldots \otimes e_{i_{n-2}}) \wedge ((1-t_0) f_{n-1} \wedge g_{n-1} \wedge g_n -t_0^{1/2} f_{n} \wedge g_{n-1} \wedge g_n)
$
\\
\\\boxed{$(3,4)$} :  $\Psi_n(\sigma_{n-1})(I_{n}(e_{i_1} \otimes \ldots \otimes e_3 \otimes e_4)) = I_n \circ (1 \otimes 1 \otimes \ldots \otimes R)(e_{i_1} \otimes \ldots \otimes e_3 \otimes e_4) \\ = t_0^{-1/2} I_{n-2}(e_{i_1} \otimes \ldots \otimes e_{i_{n-2}}) \wedge f_{n-1} \wedge f_n \wedge g_n
$
\\
\\\boxed{$(4,3)$} :  $\Psi_n(\sigma_{n-1})(I_{n}(e_{i_1} \otimes \ldots \otimes e_4 \otimes e_3)) = I_n \circ (1 \otimes 1 \otimes \ldots \otimes R)(e_{i_1} \otimes \ldots \otimes e_4 \otimes e_3) \\ = I_{n-2}(e_{i_1} \otimes \ldots \otimes e_{i_{n-2}}) \wedge (t_0^{-1/2} f_{n-1} \wedge f_n \wedge g_{n-1} + (1-t_0^{-1}) f_{n-1} \wedge f_n \wedge g_n)
$
\\
\\\boxed{$(2,4)$} :  $\Psi_n(\sigma_{n-1})(I_{n}(e_{i_1} \otimes \ldots \otimes e_2 \otimes e_4)) = I_n \circ (1 \otimes 1 \otimes \ldots \otimes R)(e_{i_1} \otimes \ldots \otimes e_2 \otimes e_4) \\ = t_0^{-1/2} (-1)^{\mu(e_{i_1} \otimes \ldots \otimes e_{i_{n-2}})} I_{n-2}(e_{i_1} \otimes \ldots \otimes e_{i_{n-2}}) \wedge f_{n-1}
$
\\
\\\boxed{$(4,2)$} :  $\Psi_n(\sigma_{n-1})(I_{n}(e_{i_1} \otimes \ldots \otimes e_4 \otimes e_2)) = I_n \circ (1 \otimes 1 \otimes \ldots \otimes R)(e_{i_1} \otimes \ldots \otimes e_4 \otimes e_2) \\ = (-1)^{\mu(e_{i_1} \otimes \ldots \otimes e_{i_{n-2}})} I_{n-2}(e_{i_1} \otimes \ldots \otimes e_{i_{n-2}}) \wedge ((1- t_0^{-1}) f_{n-1} + t_0^{-1/2} f_n)
$
\\
\\\boxed{$(1,4)$} :  $\Psi_n(\sigma_{n-1})(I_{n}(e_{i_1} \otimes \ldots \otimes e_1 \otimes e_4)) = I_n \circ (1 \otimes 1 \otimes \ldots \otimes R)(e_{i_1} \otimes \ldots \otimes e_1 \otimes e_4) \\ = - (-1)^{\mu(e_{i_1} \otimes \ldots \otimes e_{i_{n-2}})} I_{n-2}(e_{i_1} \otimes \ldots \otimes e_{i_{n-2}}) \wedge f_{n-1} \wedge g_n
$
\\
\\\boxed{$(2,3)$} :  $\Psi_n(\sigma_{n-1})(I_{n}(e_{i_1} \otimes \ldots \otimes e_2 \otimes e_3)) = I_n \circ (1 \otimes 1 \otimes \ldots \otimes R)(e_{i_1} \otimes \ldots \otimes e_2 \otimes e_3) \\ = (-1)^{\mu(e_{i_1} \otimes \ldots \otimes e_{i_{n-2}})} I_{n-2}(e_{i_1} \otimes \ldots \otimes e_{i_{n-2}}) \wedge (- f_{n-1} \wedge g_{n-1} - Y f_{n-1} \wedge g_n)
$
\\
\\\boxed{$(3,2)$} :  $\Psi_n(\sigma_{n-1})(I_{n}(e_{i_1} \otimes \ldots \otimes e_3 \otimes e_2)) = I_n \circ (1 \otimes 1 \otimes \ldots \otimes R)(e_{i_1} \otimes \ldots \otimes e_3 \otimes e_2) \\ = (-1)^{\mu(e_{i_1} \otimes \ldots \otimes e_{i_{n-2}})} I_{n-2}(e_{i_1} \otimes \ldots \otimes e_{i_{n-2}}) \wedge (-Y f_{n-1} \wedge g_n - f_n \wedge g_n)
$
\\
\\\boxed{$(4,1)$} :  $\Psi_n(\sigma_{n-1})(I_{n}(e_{i_1} \otimes \ldots \otimes e_4 \otimes e_1)) = I_n \circ (1 \otimes 1 \otimes \ldots \otimes R)(e_{i_1} \otimes \ldots \otimes e_4 \otimes e_1) \\ = (-1)^{\mu(e_{i_1} \otimes \ldots \otimes e_{i_{n-2}})} I_{n-2}(e_{i_1} \otimes \ldots \otimes e_{i_{n-2}}) \wedge (-Y f_{n-1} \wedge g_{n-1} - f_n \wedge g_{n-1} - Y^2 f_{n-1} \wedge g_n -Y f_n \wedge g_n)
$ \\
\\
Which ends the proof.
\end{proof}
\nopagebreak
  \vskip.5cm
  \par\noindent \textsc{Acknowledgments.}
I would like to thank my thesis advisor Emmanuel Wagner very warmly for his kind help and his sound advice, and for pointing me to the right direction in times of doubt.
\vskip.5cm
\end{document}